\documentclass{amsart}

\usepackage{ifthen}
\usepackage{graphicx}
\usepackage{amssymb}
\usepackage{enumerate}
\usepackage{url}
\usepackage{paralist}
\usepackage{charter}
\usepackage{natbib}
\usepackage{stmaryrd}
\usepackage{mathabx}
\usepackage{bussproofs}
\usepackage{proof}

\newtheorem{anyprop}{Anyprop}[section]

\newtheorem{theorem}[anyprop]{Theorem}
\newtheorem{lemma}[anyprop]{Lemma}
\newtheorem{proposition}[anyprop]{Proposition}
\newtheorem{corollary}[anyprop]{Corollary}

\theoremstyle{definition}

\newtheorem{definition}[anyprop]{Definition}

\newtheorem{remark}[anyprop]{Remark}

\theoremstyle{remark}

\numberwithin{equation}{section}

\usepackage{amscd}
\usepackage{amssymb}
\input xy
\xyoption{all}

\begin{document}
\title[FUNCTIONAL CONCEPTUAL SUBSTRATUM IN MATHEMATICAL CREATION]
{FUNCTIONAL CONCEPTUAL SUBSTRATUM AS A NEW COGNITIVE MECHANISM FOR MATHEMATICAL CREATION}

\author[Danny Arlen de Jes\'us G\'omez-Ram\'irez]{Danny Arlen de Jes\'us G\'omez-Ram\'irez}
\author[Stefan Hetzl]{Stefan Hetzl}
\address{Vienna University of Technology, Institute of Discrete Mathematics and Geometry,
wiedner Hauptstrasse 8-10, 1040, Vienna, Austria.}
\email{daj.gomezramirez@gmail.com}
\email{stefan.hetzl@tuwien.ac.at}

\begin{abstract}
We describe a new cognitive ability, i.e., functional conceptual substratum, used implicitly in the generation of several mathematical proofs and definitions. Furthermore, we present an initial (first-order) formalization of this mechanism together with its relation to classic notions like primitive positive definability and Diophantiveness. Additionally, we analyze the semantic variability of functional conceptual substratum when small syntactic modifications are done. Finally, we describe mathematically natural inference rules for definitions inspired by functional conceptual substratum and show that they are sound and complete w.r.t.\ standard calculi.


\end{abstract}
\maketitle

\smallskip

\section*{Introduction}
During the last decades outstanding interdisciplinary research has emerged involving the identification and subsequently formalization of the most basic cognitive mechanisms used by the mind during mathematical invention/crea- tion. Among these processes one can mention formal conceptual blending \cite{gomezothers}, \cite{fauconnierturner}; analogical reasoning \cite{gickholyoak}, \cite{hdtp}; and metaphorical thinking \cite{lakoffjohnson}, \cite{lakoff}, among others.

Now, a fundamental question related with the cognition in mathematical research involves the description of a global taxonomy of the cognitive mechanisms used (for instance) by working mathematicians for creating/inventing new mathematical results. 

So, in this paper we present an additional cognitive ability, called \emph{formal conceptual substratum}, used frequently and implicitly in the construction of mathematical arguments and definitions. We support our presentation by a significant amount of examples. Additionally, we show an initial (first-order) formalization of this mechanism and its relation with classic notions like primitive positive definability, recursive enumerability and Diophantiveness. In addition, we analyze how strongly the semantic range of this meta-notion varies (or not) when gradual changes are done to the language and to the formal structures in consideration. Finally, we present natural inference rules, insipired by functional conceptual substratum, and prove them sound and complete w.r.t.\ standard calculi.

\section{Taking inspiration from Examples}

Suppose that one should solve the following elementary question: 

 Why when we add two even (integer) numbers the result is again an even number? 

This seems to be true for small pairs of numbers $2+6=8$, $12+18=30$ and $214+674=888$. Now, for getting a general proof of this fact, we should consider syntactic expressions which can allow us to `represent' the even numbers in a compact way. Therefore, we typically come up with a (mental) representation of the form $2\cdot n$. This means that essentially we are able to represent the collection of even numbers simultaneously with the single expression $2\cdot n$, where we assume implicitly that $n$ is an integer. On the other hand, if we know that a number $c$ can be written as $2\cdot d$, where $d$ is an integer, then by definition $c$ should be an even number. In conclusion, we have found a compact (morphological-syntactic) expression for representing every even number in a unified way.

Now, let us consider again the former question with the former representation in mind: First, we need to consider two (potentially different) even numbers, so we consider (or `imagine') a first even number $2\cdot a$ and a second one $2\cdot b$, where $a$ and $b$ are integers. Second, we sum these numbers generically, namely, we obtain the expression $2\cdot a +2\cdot b$. In addition, we check if the final syntactic expression corresponds to an even number. Thus, we try to give it the desired form $2\cdot \#$, where $\#$ is a natural number. So, we factorize the former algebraic expression and get an expression of the form $2\cdot (a+b)$. Lastly, we realize that this number has the desired form $2\cdot x$, where $x=a+b$ is an integer. In conclusion, we 'proved' an affirmative answer for the former question by performing symbolic operations on morphological generic representations for even numbers.

More generally, when someone tries to solve a mathematical problem, (s)he considers, in a lot of cases, generic representations for the `standard' elements living in the corresponding mathematical structures and, subsequently (s)he performs `symbolic computations' with these representations for solving the problem and for obtaining further insights towards a final solution.

Let us consider a second example: Let $f(x)$ be a polynomial with integer coefficients. Then the polynomial $h(x)=f(x)f''(x)$ has even degree.

A very usual way for finding a proof of this statement is by taking a syntactic formal representation for $f(x)$. Effectively, from the hypothesis we see that $f(x)$ can be explicitly written as $a_mx^m+\cdots+a_0$, where $m \in \mathbb{N}$, $a_i\in \mathbb{Z}$ and $a_m\neq0$. So, we find a representation of $f''(x)$ as $m(m-1)a_mx^{m-2}+\cdots+2a_2$. 

In conclusion, we write $h(x)$ as 

\[m(m-1)a_m^2x^{2(m-1)}+\cdots+2a_0a_2.\]

So, from this representation we verify that $h(x)$ has even degree.

These morphological-syntactic representations are the seminal tools which allow us to perform general logical inferences with single syntactic elements and, simultaneously prevent us from repeating the same kind of arguments for several specific instances of $f(x)$ varying on their degrees or coefficients.

A third example comes from linear algebra. Let us assume that we have two bases $A=\{u_1,\ldots,u_n\}$ and $B=\{v_1,\ldots,v_m\}$ for a vector space $V$. So, if we want to prove (in a standard way) that the cardinality of these two bases is the same, i.e., $m=n$, then we need to use syntactic representations of the elements of $V$ such as $\sum_{i=1}^n\alpha_iu_i$ (or $\sum_{j=1}^m\beta_jv_j$). Effectively, one of the most simple arguments consists of replacing gradually the elements of one base with the elements of the other one in such a way that the resulting finite set builds again a basis. So, one begins by writing $u_1$ in terms of the elements of $B$, i.e., $u_1=\sum_{j=1}^m\gamma_jv_j$, and, subsequently, one chooses a coefficient $\gamma_{j_1}\neq 0$ in order to obtain a expression of the form 

\[v_{j_1}=\frac{1}{\gamma_{j_1}}u_1+\sum_{j=1,j\neq j_1}^m(\frac{\gamma_j}{\gamma_{j_1}})v_j.\]

Thus, one can replace $v_{j_1}$ by $u_1$ in $B$. Now, the next steps go essentially in the same (symbolical) way.

Fourth, the classic Euclidean proof (by contradiction) of the existence of infinitely many prime numbers uses in its core argument a kind of global syntactic description for a number $\prod_{i=1}^np_n+1$ bigger than one, which has no prime divisors.\footnote{Here, the assumption is that there exist finitely many prime numbers denoted by $p_1,\ldots,p_n$.}

Finally, the classic proof of the fact that the cardinality of the real numbers between zero and one (i.e. $[0,1]$) is uncountable uses as seminal argument the formal existence of a real number $\lambda=\sum_{i=1}^{\infty}b_i10^{-i}$, whose explicit decimal representation was chosen based on the corresponding decimal representations of the elements of (an hypothetical enumeration) of $[0,1]$, (i.e., $a_j=\sum_{r=1}^{\infty}a_{j,r}10^{-r}$) such that for all $i\in \mathbb{N}$, $9\neq b_i\neq a_{i,i}$.\footnote{The additional condition given by $9\neq b_i$ can be added for avoiding difficulties involving the ambiguity of the decimal representation.}

 In conclusion, this kind of `generic' syntactic representation is fundamental in several mathematical areas.

So, what lies behind the former examples is simply a specific and basic cognitive ability in which our minds choose \emph{conceptual substrata} of certain mathematical notions (e.g. even numbers and polynomials in one variable with coefficient in the integers) at a suitable level of generality, and in such a way that solving the problem simultaneously for several instances of the concepts involved can be translated into formal manipulations of fixed single conceptual representations chosen in advance.

In other words, the cognitive ability of conceptual substratum can be seen as a way of identifying and effectively using the essential (e.g. proto-typical) information of a concept in order to carry out successful deductions for solving several kinds of (mathematical) problems.


Let us consider several additional examples which allow us to enhance our initial intuitions about what the substratum of a (mathematical) concept is, and about how we can get more elements towards a first precise formalization of it. As a matter of notation we will write conceptual substrata between brackets ``$[-]$'', in order to clarify that we are talking about cognitive representations of the underlying concepts and not explicitly about the concepts themselves.

So, if $D$ denotes a mathematical concept (e.g., even numbers, polynomials, matrices, vector spaces), then we will denote by $CS(D)$ a conceptual substratum of $D$. It is important to clarify at this point that one single concept can have several conceptual substrata depending on the way in which we express such a concept syntactically. For instance, the concept of a (positive) prime number has the following two natural definitions:

\[\pi(p)=(\forall d \in \mathbb{N})(d|p \rightarrow (d=1 \vee d=p)),\]

or equivalently

\[\pi(p)=(\forall a,b \in \mathbb{N})(p|a\cdot b \rightarrow (p|a \vee p|b)).\]

From these notions one can obtain two conceptual substrata as follows:

\[CS(\text{Prime Numbers})=[d\in \mathbb{N},d|p \rightarrow (d=1 \vee d=p)],\]

and

\[CS(\text{Prime Numbers})=[a,b\in \mathbb{N},p|a\cdot b \rightarrow (p|a \vee p|b)].\]

Now, if one wishes to capture the essence of the notion of a prime number through an expression given by a term instead of the former expressions given by formulas, one can use a result of Ruiz \cite{ruiz} (among others) in order to find a quite explicit substratum for being a prime number:

\[CS(\text{Primes})=\left[1+\sum_{k=1}^{2(\left\lfloor n {\rm ln}n\right\rfloor+1)}\left(1-\left\lfloor \frac{\sum_{j=2}^k1+\left\lfloor \frac{-\sum_{s=1}^j\left(\left\lfloor \frac{j}{s}\right\rfloor-\left\lfloor \frac{j-1}{s}\right\rfloor-2\right)}{j}\right\rfloor}{n}\right\rfloor\right)\right].\] 

Most of the former conceptual substrata were expressions describing terms. Nonetheless, there are also a whole collection of concepts whose substrata are typically syntactic descriptions of relations, e.g. the number-theoretic concept of perfect number \cite{apostol}. Effectively, for this concept we can write 

\[CS(\text{Perf. Numbers})=\left[2\cdot n= \sum_{(d|n), (d>0)}d,n\in \mathbb{N}\right].\]

Another enlightening example is the concept of 'representation of the natural numbers in base $m$ ($m \in \mathbb{N}$)'.
Here we get 

\[CS(\text{m-ary Rep.})=\left[\sum_{i=0}^n\alpha_im^i:m\in\mathbb{N},\alpha_i\in\mathbb{N}, 0\leq\alpha_i<m\right].\]

We write the minimal amount of syntactic information that is required for recovering the essential features of this kind of representation.

Our approach has some informal similarities to the one based on (proto-)typicality presented in \cite{osherson}. In fact, finding the formal substratum of a concept can be seen as trying to present explicitly a morphological mathematical description of arbitrary instances of the corresponding concept, by starting with the typical ones. For instance, in our second example related with polynomials with coefficient into the integers, one can say that an expression of the form $\sum_{i=0}^2c_iX_i=c_0+c_1x+c_2x^2$, is a more typical instance of a polynomial than a constant $c_0$, or a monomial $x^m$, since the first one uses the whole spectrum of potential operations which constitute a polynomial (e.g., addition, multiplication and exponentiation), and the last ones use at most one of them. Effectively, the description of the quadratic polynomial resembles the formal substratum of the ring of polynomials better than constants or monomials.

\section{Towards a First Formalization}

As argued above, the ability to represent an arbitrary object having a certain property in a syntactic-morphological way plays a key role from a cognitive point of view. From a logical point of view this means that we are dealing with a definition {\em by a term}, or, in the case of an $r$-ary property, by a tuple of terms. Such a conceptual substratum will be called {\em functional conceptual substratum}. Let us fix a first-order logic language $L$ and an $L-$structure $M$. Now, taking inspiration from some of the former examples we state the following definition:

\begin{definition}
We say that a concept defined by a (r-ary) property $\Omega$ in $M$ (i.e. $\Omega \subseteq M^r$) has a \emph{functional conceptual substratum}, if there exist terms $t_i$ (for $i=1,\ldots,r$) and atomic formulas $A_1,\ldots,A_m$ whose variables are contained in  $\{x_1,\ldots,x_n\}$, such that for all $a_1,\ldots,a_r\in M$, $(a_1,\ldots,a_r)\in \Omega$ if and only if
\[
M\models (\exists x_1)\cdots(\exists x_n)(a_1=t_1\wedge\cdots\wedge a_r= t_r\wedge A_1\wedge \cdots \wedge A_m)
\]
where $t_1,\ldots,t_n$ are $L$-terms whose variables are among $x_1,\ldots,x_n$.
\end{definition}
So, it is straightforward to verify that the notions of even, odd and compose numbers; perfect squares and (more generally) nth-powers have functional conceptual substrata.

In addition one can prove that this notion coincides with primitive positive definability (see for example \cite{bodirsky}).

\subsection{Classic Arithmetic Structures}

Now, let us see how the fact that having this kind of `functional conceptual representations' materializes for several language-structure combinations.

First, it is worth noting that if we do not put any additional restriction on the atoms $A_j$ in the former definition, then for some $\Omega$ it could happen that these atoms contain even more important information about the concept $C$ than the terms $t_i$, for $i=1,\ldots,n$. Later, we will show explicitly this phenomenon with an example.

Let us consider the language $L=\left\{0,1,+,-,*,=,<\right\}$ and the structure $\mathbb{Z}$, the integers. Then, each $A_j$ has the form of either $u_1(x_1,\ldots,x_n)=u_2(x_1,\ldots,x_n)$ or $u_1(x_1,\ldots,x_n)<g_2(x_1,\ldots,u_n)$, where $u_1$ and $u_2$ are the corresponding polynomials in $\mathbb{Z}[x_1,\ldots,x_n]$ representing the terms appearing in $A_j$.

Now, in the first case $A_j$ can be rewritten as $h(x_1,\ldots,x_n)=0$, where $h=g_1-g_2$.
For the second case, we can use the well-known fact that any natural number can be written as the sum of four perfect squares \cite{hardywright} (i.e. Lagrange's theorem) in order to express the condition described by $A_j$ in a Diophantine way, i.e.,  

\[(\exists z_1 \cdots z_4)(g_1-g_2=z_1^2+z_2^2+z_3^2+z_4^2+1).\]

In addition, one can also express finite conjunctions of polynomial equations through a single equation by using the fact that over the integers it holds that $\sum_{i=1}^na_i^2=0$ if and only if each $a_i=0$. So, combining all the former steps one can construct an explicit polynomial $H(y_1,\ldots,y_r,x_1,\ldots,x_w)$ such that 
 
\[\mathbb{Z}\models (\forall y_1 \cdots y_r)(\Omega(y_1,\ldots,y_r) \leftrightarrow (\exists x_1 \cdots x_w)(H(y_1,\ldots,y_r,x_1,\ldots,x_w)=0)).\]

In other words, $\Omega$ defines a Diophantine set \cite[Ch. 1]{matiyasevich}.

Furthermore, by the MRDP theorem \citep[Ch. 2]{matiyasevich} $\Omega$ defines a recursively enumerable set. In fact, one can easily prove that a concept $C$ over the integers, described with the former language which has an functional conceptual substratum, must be recursively enumerable. Conversely, if $C$ is a concept defining a recursively enumerable property $\Theta$, then by the MRDP theorem $\Theta$ is Diophantine. Thus, for all $a_1\ldots,a_r\in\mathbb{Z}$,  $a_1,\ldots,a_r\in\Theta$ if and only if

\[\mathbb{Z}\models (\exists x_1 \cdots x_m)(F(a_1,\ldots,a_r,x_1,\ldots,x_m)=0)).\]

We can rewrite this formula as

\[\mathbb{Z}\models (\exists x_1 \cdots x_m)(\exists x'_1 \cdots x'_r)(a_1=x'_1\wedge\cdots\wedge a_r=x'_r\wedge A_1))\]

where $A_1$ denotes the atom $F(x'_1,\ldots,x'_r,x_1,\ldots,x_m)=0))$.\footnote{In this case, the essential information of the concept can be, at least formally, codified more in the atom $A_1$ rather than in the initial polynomial expressions.}

In conclusion, for $\mathbb{Z}$ expressed in the language $L=\left\{0,1,+,-,*,=,< \right\}$ a concept $C$ describing an $n-$ary property $\Omega$ has functional conceptual substratum if and only if $\Omega$ is recursively enumerable, which is equivalent to being Diophantine.

\subsection{The Notion of a Prime Number}

By the former considerations, the set of prime numbers has an functional conceptual substratum. More explicitly, one can find an explicit polynomial inequality in the integers characterizing the positive prime numbers. For example, based on the main result of \cite{jonesetal} we can describe an (atomic) conceptual substratum of the prime numbers as follows

\[CS(\text{Prime Numbers})=[k,(k+2)(1-(wz+h+j-q)^2\]

\[-((gk+2g+k+1)(h+j)+h-z)^2-(2n+p+q+z-e)^2-\]

\[(16(k+1)^3(k+2)(n+1)^2+1-f^2)^2-(e^3(e+2)(a+1)^2+1-o^2)^2\]

\[-((a^2+1)y^2+1-x^2)^2-(16r^2y^4(a^2-1)+1-u^2)^2\]

\[(((a+u^2(u^2+a))^2-1)(n+4dy)^2+1-(x+cu)^2)^2-(n+l+v-y)^2\]

\[-((a^2-1)l^2+1-m^2)^2-(ai+k+1-l-i)^2\]

\[-(p+l(a-n-1)+b(2an+2a-n^2-2n-2)-m)^2\]

\[-(q+y(a-p-1)+s(2ap+2a-p^2-2p-2)-x)^2\]

\[-(z+pl(a-p)+t(2ap-p^2-1)-mp)^2)>0,\]

\[b,c,d,e,f,g,h,i,h,j,l,m,n,p,q,r,s,t,u,v,w,x,y,z\in \mathbb{N}]\]

In addition, by Lagrange's theorem and by adding four new existential quantified variables replacing each of the former 26 variables, one can show that there exists a polynomial $P(x_1,\ldots,x_{108})$ with integer coefficients, such that 

\[CS(\text{Prime Numbers})=[x_1,P(x_1,\ldots,x_{108})>0,x_2,\ldots,x_{108}\in \mathbb{Z}]\]

So, the concept of prime numbers has an functional conceptual substratum over $\mathbb{Z}$ described in the former language.

Now, let us focus on the question of deciding if the concept of prime numbers has an functional conceptual substratum where the atoms $A_j$ have either the form $x_{r_j}<c_j$ or $c_j<x_{c_j}$.

So, essentially this question is equivalent to finding a polynomial with integer coefficients $f(x_1,\ldots,x_n)$ such that the set of the prime numbers is generated as the image of the domain defined by the atomic restrictions $A_1,\ldots,A_m$. Let us prove by induction on $n$ that this cannot happen. 

First, let us suppose that $f(x)$ is a polynomial in one variable with restrictions given by $A_1\cong x<c_1$ and/or $A_2\cong c_2<x$. The case where the domain is either empty or finite (parametrized by two atoms) is clearly ruled out, since its image should be an infinite set. The single cases given by just one of the former atoms can be reduced to the case $x>c_1$, because the second case can be reduced to this one by means of the change of variables $y=-x$. 

In conclusion, let us assume by the sake of contradiction that there exists a polynomial $f(x)$ with integer coefficients together with a constant $c\in \mathbb{Z}$ such that the image under $f$ of the set $\mathbb{Z}_{>c}$ is the set of the prime numbers (or an infinite subset of it). 
Let us choose an integer $d>c$. If we denote by $p$ the prime number $f(d)$, it is an elementary fact to see that for all $z\in \mathbb{Z}$ 

\[f(pz+d)\equiv f(d)\equiv 0\ ({\rm mod}\ p).\]

Thus, since $f(pz+d)$ should be a prime number for all $z\geq 0$, then $f(pz+d)=p$. Therefore, $f$ should be a constant polynomial, which is a contradiction.

Now, let us assume the induction's hypothesis for any $k< n$. Again, suppose by contradiction that there exists a polynomial $f(x_1,\ldots,x_n)$ with coefficients in the integers and atoms (restrictions) $A_1,\ldots,A_m$ such that the image of the domain determined by the restrictions consists of (an infinite subset of) the prime numbers. Again, by doing suitable changes and permutations of variables we can assume without loss of generality that there exists $s\in \mathbb{Z}$ with $1\leq s \leq m$, and constants $c_i\in\mathbb{Z}$ such that $A_i\cong x_i>c_i$, for all $i=1,\ldots,s$. Thus, since there are just finitely many potential choices for the values of the $x_i$'s (with $i>s$) which satisfy the restrictions, we see that there are constants $e_{s+1},\ldots,e_n\in \mathbb{Z}$ satisfying all the remaining conditions $A_{s+1},\ldots,A_m$, such that the image of the domain described by the first $s$ atomic restrictions under the polynomial 

\[g(x_1,\ldots,x_s)=f(x_1,\ldots,x_s,e_{s+1},\ldots,e_m)\]

 is an infinite subset of the prime numbers. So, if $s<m$ we are done by the induction's hypothesis, since $g$ has fewer variables than $f$. 

In the second case, it is an elementary fact to see that for any non-constant polynomial $g(x_1,\ldots,x_s)$ in several variables with integer (or even real) coefficients, and for any parameters $c_1,\ldots,c_s\in \mathbb{R}$ (defining atomic restriction as before), there exists an index $i_1$ and an integer (resp. a real number) $e>c_{i}$ such that $h=f(x_1,\ldots,x_{i-1},e,x_{i+1},\ldots,x_s)$ is a non-constant polynomial. 

Now, using this fact, we obtain a non-constant polynomial $h$ in $s-1$ variables, such that the image of the remaining restrictions under $h$ is an infinite subset of the prime numbers, which is a contradiction.

Summarizing, the existence of functional conceptual representations depends strongly on the degree of freedom that we give to the corresponding atomic formulas.

On the other hand, let us modify the language slightly by trying to characterize the prime numbers as a kind of `sub-concept' of the natural numbers $\mathbb{N}$ with the language $L^-=\left\{0,1,+,*,=,<\right\}$, and with the former constrains for the atoms $A_i$. So, by applying basically the same method as before, we obtain again a negative answer.

However, if we do not impose any kind of restriction on the atoms, then using the same former result of Jones et at. one can find two explicit polynomial $P_1(a,b,\ldots,z)$ and $P_2(a,b,\ldots,z)$ with coefficients into the natural numbers such that 
\[ CS(\text{Prime Numbers})=[k,P_1(a,b,\ldots,z)>P_2(a,b,\ldots,z),b,c,\ldots,z]\]
So, the notion of prime numbers has also an functional conceptual substratum over $\mathbb{N}$ with the restricted language $L^-$.

\subsection{The Arithmetical Invariance of Functional Conceptual Substratum}

More generally, if we restrict ourselves to a concept $C$ described by a $n-$ relation in $\mathbb{N}$, then the fact that $C$ has an functional conceptual substratum does not change if we expand the language involved (resp. the corresponding structure) by adding the operation $-()$. Specifically, the following general fact holds:

\begin{proposition}
Let $C$ be a concept described by a $r-$ary relation $D$ in $\mathbb{N}$. Then $C$ has an functional conceptual substratum in $L^-$, if and only if $C$ (seen as a concept described by the corresponding $r-$ary relation $D\subseteq \mathbb{Z}$) has an functional conceptual substratum in $L$.
\end{proposition}

\begin{proof}
Without loss of generality, we can assume that $r=1$ (the general argument is essentially the same).
First, let us suppose that there is an $L-$functional conceptual substratum for $C$ involving $f(x_1,\ldots,x_n)$ and atoms $A_1,\ldots,A_r$. Now, we will add an extra variable $z$ in order to be able to codify the fact that $a_1=f(x_1,\ldots,x_n)$ through the atoms $a_1=z$ and $A_{r+1}\equiv z=f(x_1,\ldots,x_n)$. This allows us to update $f$ by a polynomial with positive coefficients.

By Lagrange's theorem and by adding (eventually) new existential quantified variables, we can assume that all the atoms involve only the equality relation. Effectively, this follows from the relations

\[(\forall a,b\in \mathbb{Z})(a<b \leftrightarrow a+1\leq b),\]

\[(\forall c,d\in \mathbb{Z})(c\leq d \leftrightarrow(\exists y_1,y_2,y_3,y_4 \in \mathbb{Z})(d-c=\sum_{i=1}^4y_i^2)).\]

An additional simplification consists in reducing the number of atoms to one, by using the fact that 

\[(\forall e,g \in \mathbb{Z}((e=0 \wedge g=0) \leftrightarrow e^2+g^2=0)).\] 
So, let us assume the we have just one atom $A$.

 Furthermore, the fact that there are existential conditions for $A$ involving variables $z$ and 
$x_1,\ldots,x_n$ varying over $\mathbb{Z}$, can be re-written as new atom $A'$ involving variables $z'$ and $x'_1,\ldots,x'_n$ varying now over $\mathbb{N}.$

In fact, if $A\equiv h_1(z,x_1,\ldots,x_n)=h_2(z,x_1,\ldots,x_n)$, then the fact that there exists $z,x_1,\ldots,x_n\in \mathbb{Z}$ such that $A(z,x_1,\ldots,x_n)$ is equivalent to say that that there exist $z',x'_1,\ldots,x'_n\in \mathbb{N}$ such that 

\[\bigvee (h_1(\pm z,\pm x'_1,\ldots,\pm x'_n)=h_2(\pm z,\pm x'_1,\ldots, \pm x'_n)),\]

where the former expression involves $2^{n+1}$ atoms corresponding to all the possible combinations of signs. 
Now, by writing each of the former equalities as $\varphi_j(\underline{x'})=0$, for $j=1,\ldots,2^{n+1}$, we can re-write the former expression as the single atomic condition

\[\Phi(z',\underline{x'})=\prod_{j=1}^{2^{n+1}}\varphi_j(z',\underline{x'})=0.\]

Finally, we can re-write this condition as a polynomial equality of the form $\gamma_1(z',\underline{x'})=\gamma_2(z',\underline{x'})$ involving only positive coefficients.

So, for all $a\in \mathbb{N}$, $a\in D$ if and only if 

\[(\exists w_1 \cdots w_{n+1})(a=w_1\wedge \gamma_1(w_1,\ldots,w_{n+1})=\gamma_2(w_1,\ldots,w_{n+1})).\]

This means that $C$ has $L-$functional conceptual substratum.

Conversely, we replace in a $L^--$functional conceptual subtratum, any variable $x_j$ by four variables $y_{j,1},y_{j,2},y_{j,3}$ and $y_{j,4}$; and we replace each occurrence of $x_j$ by $\sum_{i=1}^4y_i^2$. So, by Lagrange's theorem, we obtain an $L-$functional conceptual substratum for $C$.
\end{proof}

\begin{remark}
If we replace in the former proposition functional conceptual substratum by Diophantine, then the answer is quite different. Effectively, by the MRDP theorem we know that the `Diophantine' $L-$concepts are exactly the recursively enumerable. However, the set of Diophantive $L^--$concepts corresponds to a strictly smaller sub-collection of them. Specifically, if $r=1$, then it is an elementary exercise to prove that the only two Diophantine $L^--$subsets of $\mathbb{N}$ (i.e. subsets described as projections over $\mathbb{N}$ of a polynomial with non-negative coefficients) are $\{0\}$ and $\mathbb{N}$. In general, one can verify by induction over $r$ that a subset $\Omega\subseteq\mathbb{N}^r$ is $L^--$Diophantine if it has the form

\[\bigcup_{\underline{i}=(i_1,\ldots,i_k)}^{\text{finite}}\prod_{r=1}^k\mathbb{N}^{(i_r)},\]

where $i_s\in\{0,1\}$ and we define $\mathbb{N}^0=\{0\}$ and $\mathbb{N}^1=\mathbb{N}$.
\end{remark}

\section{Natural and Complete Definition Rules for Functional Conceptual Substratum}

Let us denote by $\mathrm{LK}_\mathrm{e}$ the sequent calculus for first-order predicate logic with equality (over a language $L$) with the standard inference rules (see for instance \cite{buss}, \cite{takeuti}). Let us enlarge the language $L$ with a new $r-$ary predicate symbol $D$ which we will define
in terms of a functional conceptual substratum in the language $L$, i.e., by a definition of the form
\[D(a_1,\ldots,a_r) \Leftrightarrow  (\exists x_1 \cdots x_n)(a_1=t_1\wedge\cdots\wedge a_r = t_r\wedge
 A_1\wedge \cdots \wedge A_m)\]
where $t_1,\ldots,t_n$ are $L-$terms and $A_1,\ldots,A_m$ are $L$-atoms whose variables are (both) among $x_1,\ldots,x_n$.

Now, a standard approach to incorporate definitions into a sequent calculus is to add definition rules which
allow unfolding the defined predicate symbol. In our setting this gives rise to the rules
\[
\begin{array}{c}
\infer[D_\mathrm{L}]{D(a_1,\ldots,a_r),\Gamma \rightarrow \Delta}{\phi(a_1,\ldots,a_r),\Gamma\rightarrow \Delta}
\end{array}
\qquad\text{and}\qquad
\begin{array}{c}
\infer[D_\mathrm{R}]{\Gamma \rightarrow \Delta,D(a_1,\ldots,a_r)}{\Gamma \rightarrow \Delta,\phi(a_1,\ldots,a_r)}
\end{array}
\]
where  $\phi(a_1,\ldots,a_r)$ abbreviates the formula defining $D(a_1,\ldots,a_r)$ as above. We denote the sequent
calculus obtained from adding these rules to $\mathrm{LK}_\mathrm{e}$ as $\mathrm{LK}_\mathrm{e}(D)$.
These rules correspond to inferences that syntactically replace into a proof the former definition of the new relational symbol within the left and right part of a sequent, respectively.

\begin{lemma}\label{lem.LKeD}
For any formula $\psi$, $\mathrm{LK}_\mathrm{e}(D) \vdash \psi \leftrightarrow \psi[D\backslash \phi]$, where $\psi[D\backslash \phi]$ denotes the formula obtained after replacing $D$ by $\phi$ in $\psi$.
\end{lemma}

\begin{proof}
This fact can be straightforwardly proved by induction on the (syntactic) complexity of $\psi$, decomposing the equivalence into two implications and using the new pair of rules.
\end{proof}

The calculus $\mathrm{LK}_\mathrm{e}(D)$ is a conservative extension of $\mathrm{LK}_\mathrm{e}$ in
the following sense:
\begin{theorem}
For any formula $\psi$, $\mathrm{LK}_\mathrm{e}(D)\vdash \psi$ if and only if $\mathrm{LK}_\mathrm{e} \vdash \psi[D\backslash \phi]$.
\end{theorem}

\begin{proof}
$(\Rightarrow)$
Let $P$ be a proof of $\psi$ in $\mathrm{LK}_\mathrm{e}(D)$. Then, by replacing $D$ in $P$ by $\phi$ and
removing $D_\mathrm{L}$- and $D_\mathrm{R}$-inferences, we obtain a proof $P'$ of $\psi[D\backslash \phi]$
in $\mathrm{LK}_\mathrm{e}$.

$(\Leftarrow)$
Let $P$ be an $\mathrm{LK}_\mathrm{e}$-proof of $\psi[D\backslash \phi]$. Obtain an $\mathrm{LK}_\mathrm{e}(D)$-proof $Q$ of $\psi[D\backslash \phi] \rightarrow \psi$ from Lemma~\ref{lem.LKeD}. Then a cut on $P$ and $Q$ gives an $\mathrm{LK}_\mathrm{e}(D)$-proof of $\psi$.
\end{proof}

The above definition rules treat definitions in general. However, a definition of a concept that has a functional conceptual substratum is typically used in a more specific way in mathematical proofs. For example, when showing
that the sum of $n$ and $m$ is even if $m$ and $n$ are, one may start the proof by a phrase like ``Since $n$ is even, $n=2a$ (for some $a\in\mathbb{N}$)''. For the general case, this is formalized by the rule
\begin{prooftree}
\AxiomC{$a_1=t_1[\underline{x}\backslash\underline{\zeta}],\ldots,a_r=t_r[\underline{x}\backslash\underline{\zeta}],A_1[\underline{x}\backslash\underline{\zeta}],\ldots,A_m[\underline{x}\backslash\underline{\zeta}], \Gamma \rightarrow \Delta$}
\RightLabel{$D^\mathrm{fcs}_\mathrm{L}$}
\UnaryInfC{$D(a_1,\ldots,a_r),\Gamma \rightarrow \Delta$}
\end{prooftree}
Similarily, one may end the proof with a phrase like ``$2\cdot(a+b)$ is even''. For the general case, this is formalized by the rule
\[
\infer[D^\mathrm{fcs}_\mathrm{R}]{\Gamma \rightarrow \Delta, D(t_1[\underline{x}\backslash\underline{u}],\ldots,t_r[\underline{x}\backslash\underline{u}])}{
  \Gamma \rightarrow \Delta, A_1[\underline{x}\backslash\underline{u}]
  &
  \cdots
  &
  \Gamma \rightarrow \Delta, A_m[\underline{x}\backslash\underline{u}]
}
\]

We write $\mathrm{LK}^\mathrm{fcs}_\mathrm{e}$ for the calculus obtained from $\mathrm{LK}_\mathrm{e}$ by
adding these two rules. We will now verify that $\mathrm{LK}^\mathrm{fcs}_\mathrm{e}(D)$ is sound and complete w.r.t.\ $\mathrm{LK}_\mathrm{e}(D)$.
To that aim, we first relate it to $\mathrm{LK}_\mathrm{e}$.

\begin{lemma}\label{lem.LKeDfcs}
For any formula $\psi$, $LK_\mathrm{e}^\mathrm{fcs}(D) \vdash \psi \leftrightarrow \psi[D\backslash \phi]$.
\end{lemma}

\begin{proof}
We proceed by induction on the syntactic complexity of $\psi$. The only non-trivial case is when $\psi$ is $D(v_1,\ldots,v_r)$.

We obtain an $\mathrm{LK}_\mathrm{e}^\mathrm{fcs}(D)$-proof of $D(v_1,\ldots,v_r) \rightarrow \phi(v_1,\ldots,v_r)$ by applying a $D_\mathrm{L}^\mathrm{fcs}$-inference,
$n$ $\exists_\mathrm{r}$-inferences, and $r+m-1$ $\wedge_\mathrm{r}$-inferences.

In the other direction, we obtain an $\mathrm{LK}_\mathrm{e}^\mathrm{fcs}(D)$-proof of $\phi(v_1,\ldots,v_r) \rightarrow D(v_1,\ldots,v_r)$ as follows:
\[
\infer[\exists_\mathrm{l}^n]{\phi(v_1,\ldots,v_r) \rightarrow D(v_1,\ldots,v_r)}{
  \infer[\wedge_\mathrm{l}^{r+m-1}]{v_1 = t_1[\underline{x}\backslash\underline{\zeta}]\wedge \cdots \wedge v_r = t_r[\underline{x}\backslash\underline{\zeta}] \wedge A_1[\underline{x}\backslash\underline{\zeta}] \wedge \cdots \wedge A_m[\underline{x}\backslash\underline{\zeta}] \rightarrow D(v_1,\ldots,v_r)}{
    \infer[\mathrm{eq.}]{v_1 = t_1[\underline{x}\backslash\underline{\zeta}], \ldots, v_r = t_r[\underline{x}\backslash\underline{\zeta}], A_1[\underline{x}\backslash\underline{\zeta}],\ldots,A_m[\underline{x}\backslash\underline{\zeta}] \rightarrow D(v_1,\ldots,v_r)}{
      \infer[D_\mathrm{R}^\mathrm{fcs}]{A_1[\underline{x}\backslash\underline{\zeta}],\ldots,A_m[\underline{x}\backslash\underline{\zeta}]\rightarrow D(t_1[\underline{x}\backslash\underline{\zeta}],\ldots,t_r[\underline{x}\backslash\underline{\zeta}])}{
        A_1[\underline{x}\backslash\underline{\zeta}] \rightarrow A_1 [\underline{x}\backslash\underline{\zeta}]
        &
        \cdots
        &
        A_m[\underline{x}\backslash\underline{\zeta}] \rightarrow A_m[\underline{x}\backslash\underline{\zeta}]
      }
    }
  }
}
\]
\end{proof}

\begin{theorem}
For any formula $\psi$, $\mathrm{LK}_\mathrm{e}^\mathrm{fcs}(D) \vdash \psi$ if and only if $\mathrm{LK}_\mathrm{e} \vdash \psi[D\backslash \phi]$.
\end{theorem}

\begin{proof}
$(\Rightarrow)$
Let $P$ be a proof of $\psi$ in $\mathrm{LK}^\mathrm{fcs}_\mathrm{e}(D)$. We replace $D$ in $P$ by $\phi$,
simulating a $D_\mathrm{L}^\mathrm{fcs}$-inferece by $n$ $\exists_\mathrm{l}$-inferences, and $m+r-1$ $\wedge_\mathrm{l}$-inferences
and a $D_\mathrm{R}^\mathrm{fcs}$-inference by
\[
\infer[\exists_\mathrm{r}^n]{\Gamma \rightarrow \Delta, \phi(t_1[\underline{x}\backslash\underline{u}], \ldots, t_r[\underline{x}\backslash\underline{u}])}{
  \infer[\mathrm{eq.}]{\Gamma\rightarrow\Delta, t_1[\underline{x}\backslash\underline{u}] = t_1[\underline{x}\backslash\underline{u}] \wedge \cdots \wedge t_r[\underline{x}\backslash\underline{u}] = t_r[\underline{x}\backslash\underline{u}] \wedge A_1[\underline{x}\backslash\underline{u}] \wedge \cdots \wedge A_m[\underline{x}\backslash\underline{u}]}{
    \infer[\wedge_\mathrm{r}^{m-1}]{\Gamma\rightarrow\Delta, A_1[\underline{x}\backslash\underline{u}] \wedge \cdots \wedge A_m[\underline{x}\backslash\underline{u}]}{
      \Gamma \rightarrow \Delta, A_1[\underline{x}\backslash\underline{u}]
      &
      \cdots
      &
      \Gamma \rightarrow \Delta, A_m[\underline{x}\backslash\underline{u}]
    }
  }
}
\]

Thus we obtain a proof $P'$ of $\psi[D\backslash \phi]$ in $LK_\mathrm{e}$

$(\Leftarrow)$
Let $P$ be an $\mathrm{LK}_\mathrm{e}$-proof of $\psi[D\backslash \phi]$. Obtain an $\mathrm{LK}_\mathrm{e}(D)$-proof $Q$ of $\psi[D\backslash \phi] \rightarrow \psi$ from Lemma~\ref{lem.LKeDfcs}. Then a cut on $P$ and $Q$ gives an $\mathrm{LK}_\mathrm{e}^\mathrm{fcs}(D)$-proof of $\psi$.
\end{proof}

\begin{corollary}
For any formula $\psi$, $\mathrm{LK}_\mathrm{e}^\mathrm{fcs}(D) \vdash \psi$ iff $\mathrm{LK}_\mathrm{e}(D) \vdash \psi$.
\end{corollary}

Thus one does not loose power by using these specialized definition rules for defined predicate symbols with functional conceptual substratum. On the other hand, one gains a mathematically more natural
use of these defined symbols.

\section{Conclusions}

The general meta-notion of conceptual substratum (and its particular form as functional conceptual substratum) serves as a new kind of (meta-mathema- tical) cognitive mechanism of seminal importance used (implicitly) in mathematical creation/invention.

Moreover, the initial first-order formalization of this meta-concept turns out to be equivalent to central notions in theoretical computer sciences and elementary number theory. In addition, (functional) conceptual substratum suggests an additional way of developing proof-theoretical frameworks with a stronger human-style structure. So, subsequent formalizations of conceptual substratum in higher-order frameworks could bring new light in our quest for understanding how mathematical creation/invention works and for developing software being able to solve mathematical problems at higher levels of abstraction.

\section*{Acknowledgements}

This work was supported by the Vienna Science and Technology Fund (WWTF), Vienna Research Group 12-004.
In addition, the first author wants to thank B. Kresina for all the inspiration, and to Eunise, Carlos and Jeronimo Lopera for their special support and kindness.

\bibliographystyle{plainnat}

\end{document}